\documentclass[11pt]{amsart}
\newtheorem{theorem}{Theorem}[section]
\newtheorem{corollary}[theorem]{Corollary}
\newtheorem{lemma}[theorem]{Lemma}
\newtheorem{proposition}[theorem]{Proposition}
\theoremstyle{definition}
\newtheorem{definition}[theorem]{Definition}
\theoremstyle{remark}
\newtheorem{remark}[theorem]{Remark}

\numberwithin{equation}{section}
\usepackage{url}
\usepackage{hyperref,amsthm,amssymb}
\hypersetup{colorlinks=true,linkcolor=blue,
urlcolor=cyan,citecolor=magenta} 
\usepackage[hyperpageref]{backref}
\usepackage{frcursive}
\usepackage{calrsfs}
\usepackage{mathrsfs}
\newcommand{\Q}{\mathbb{Q}}
\newcommand{\Z}{\mathbb{Z}}

\newcommand{\ds}{\displaystyle}
\newcommand{\ov}{\overline}

\newcommand{\ft}{\footnotesize}
\newcommand{\ns}{\normalsize}
\newcommand{\plus}{\ds\mathop{\raise 0.5pt \hbox{$\bigoplus$}}\limits}
\newcommand{\prd}{\ds\mathop{\raise 1.0pt \hbox{$\prod$}}\limits}
\newcommand{\sm}{\ds\mathop{\raise 1.0pt \hbox{$\sum$}}\limits}
\newcommand{\order}{\raise0.8pt \hbox{${\scriptstyle \#}$}}
\newcommand{\lien}{\mathrel{\mkern-4mu}}
\newcommand{\too}{\relbar\lien\rightarrow}
\newcommand{\tooo}{\relbar\lien\relbar\lien\too}
\newcommand{\ffrac}[2]{\hbox{\ft $\displaystyle\frac{#1}{#2}$}}
\newcommand{\Nu}{\hbox{\Large $\nu$}}
\newcommand{\Kappa}{\hbox{\large $\kappa$}}
\newcommand{\BKappa}{\hbox{\Large $\kappa$}}
\newcommand{\Gal}{{\rm Gal}}
\newcommand{\Ker}{{\rm Ker}}
\newcommand{\Coker}{{\rm Coker}}
\newcommand{\Hom}{{\rm H}}
\newcommand{\ab}{{\rm ab}}
\newcommand{\ar}{{\hspace{1 pt}\rm ar}}
\newcommand{\alg}{{\hspace{1 pt}\rm alg}}
\newcommand{\Norm}{{\bf N}}
\newcommand{\BJ}{{\bf J}}
\newcommand{\Bj}{{\bf j}}
\newcommand{\BH}{{\bf H}}

\newcommand{\BE}{{\bf E}}
\newcommand{\BF}{{\bf F}}

\newcommand{\CE}{{\mathcal E}}
\newcommand{\CF}{{\mathcal F}}
\newcommand{\CH}{{\mathcal H}}
\newcommand{\CG}{{\mathcal G}}

\newcommand{\CA}{{\mathcal A}}
\newcommand{\CB}{{\mathcal B}}
\newcommand{\CX}{{\mathcal X}}

\begin{document}

\title[On the non semi-simple Real Abelian Main Conjecture]
{On the Real Abelian Main Conjecture \\ in the non semi-simple case}

\author{Georges Gras}

\address{Franche--Comt\'e--Besan\c con University (in retirement)
4, chemin de Ch\^ateau Gagni\`ere,
F-38520 Le Bourg d'Oisans  
 \url{http://orcid.org/0000-0002-1318-4414}}
\email{g.mn.gras@wanadoo.fr}


\date{June 16, 2023}


\begin{abstract}
Let $K/\Q$ be a real cyclic extension of degree divisible by $p$. 
We analyze the {\it statement} of the ``Real Abelian Main Conjecture'',
for the $p$-class group $\CH_K$ of $K$, in this non semi-simple case. 
The classical {\it algebraic} definition of the $p$-adic isotopic components
$\CH^\alg_{K,\varphi}$, for irreducible $p$-adic characters $\varphi$,
is inappropriate with respect to analytical formulas, because of 
capitulation of $p$-classes in the $p$-sub-extension of $K/\Q$. 
In the 1970's we have given an {\it arithmetic} definition,
$\CH^\ar_{K,\varphi}$, and formulated the conjecture, still unproven,  
$\order \CH^\ar_{K,\varphi} = \order (\CE_K / \CE^\circ_K \, \CF_{\!K})_{\varphi_0}$, 
in terms of units $\CE_K$ then $\CE^\circ_K$ (generated by units of the strict subfields
of $K$) and cyclotomic units $\CF_K$, where $\varphi_0$ is the tame 
part of~$\varphi$. We prove that the conjecture holds as soon as there exists a prime $\ell$, 
totally inert in $K$, such that $\CH_K$ capitulates in $K(\mu_\ell^{})$, existence 
having been checked, in various circumstances, as a promising new tool. 
\end{abstract}

\maketitle

\vspace{-0.8cm}
\tableofcontents

\section{Introduction and preliminary remarks} 

Let $K/\Q$ be a real cyclic extension of Galois group $\CG_K \simeq 
\gamma \oplus \Gamma$, $\gamma$ of prime-to-$p$ order $d$, $\Gamma \simeq 
\Z/p^e \Z$, $e \geq 1$, and let $k$ be the subfield of $K$ fixed by $\Gamma$. 
Let $\CH_K$ be the $p$-class group of $K$, as 
$\Z_p[\CG_K]$-module and let $E_K$ (resp. $F_K$) be the group of units
(resp. of Leopoldt's cyclotomic units) of $K$ and set $\CE_K = E_K \otimes \Z_p$,
$\CF_{\!K} = F_K \otimes \Z_p$. We consider $\CE^\circ_K = E^\circ_K \otimes \Z_p$,
where $E^\circ_K$ is the subgroup of $E_K$ generated by the units of the
strict subfields of $K$.

\smallskip
The ``Real Abelian Main Conjecture'' (RAMC for short) writes:
\begin{equation}\label{RAMC}
\order \CH^\ar_{K,\varphi} = \order (\CE_K / \CE^\circ_K \, \CF_{\!K})_{\varphi_0},\ \,
\varphi = \varphi_0^{} \varphi_p
\end{equation}
($\varphi_p$ of $p$-power order, $\varphi_0^{}$, of prime-to-$p$ order, is called the 
tame part of the irreducible $p$-adic character $\varphi$ of $K$);  
RAMC claims analytic expressions of orders 
of suitable $p$-adic isotopic components $\CH^\ar_{K,\varphi}$ of the $p$-class group 
of $K$ and was first stated in the papers \cite{Gra1976} (1976), \cite{Gra1977$^a$} (1977) 
(especially for \textit {the non semi-simple case}); the semi-simple case 
$p \nmid [K : \Q]$, claiming $\order \CH_{K,\varphi} = \order (\CE_K / \CF_{\!K})_\varphi$,
was the purpose of \cite{Gra1977$^b$} (1977). These conjectures were presented at the 
meeting ``Journ\'ees arithm\'etiques de Caen'' (1976) as it is mentioned 
for instance in a paper of Solomon \cite{Sol1990} (1990), in a survey of Ribet 
\cite{Rib2008} (2008) and in Washington's book \cite[\S\,15.3]{Was1997} (1997). 
The non semi-simple RAMC is still unproven, contrary to some affirmations in the literature. 

\medskip\noindent
{\bf Main result} (Theorem \ref{relfond}).
The RAMC \eqref{RAMC} holds for $K$ as soon as there exists a prime number 
$\ell \equiv 1\! \pmod {2p^N}$, totally inert in $K$, such that $\CH_K$ capitulates 
in the auxiliary $p$-sub-extension of $K(\mu_\ell^{})/K$.

\medskip
The classical {\it algebraic} definition of the $p$-adic isotopic components, in the form 
$\CH^\alg_{K,\varphi} := \CH_K \ \hbox {$\otimes^{}_{\Z_p[\CG_K]}$} \Z_p[\mu_{d p^e}^{}]$ 
with obvious Galois action, is inappropriate with respect to analytic formulas, because of 
frequent capitulations of $p$-classes in the $p$-sub-extension $K/k$ of $K/\Q$. In our 
previous works, we have given, for any rational character $\chi$ of $K$ and 
$\varphi \mid \chi$, an {\it arithmetic} definition of these $\chi$ and $\varphi$-components, 
denoted $\CH^\ar_{K,\chi}$ and $\CH^\ar_{K,\varphi}$, giving the expected formulas 
$\order \CH^\ar_{K,\chi} = \prod_{\varphi \mid \chi} \order \CH^\ar_{K,\varphi}$ (Theorem
\ref{isotopicphi}). 
For convenience, we will refer to the english translation of \cite{Gra1976} (1976) 
given in \cite{Gra2021} (2021). 

\smallskip
The RAMC has been proven in some semi-simple cases by means of an impressive 
series of articles (beginning with the geometrical methods of Ribet, Mazur--Wiles, Rubin,
then ending by means of Kolyvagin's Euler Systems from Thaine's arithmetic 
approach), then in the non semi-simple case for {\it relative class groups} of 
imaginary abelian fields by Solomon and Greither (in the imaginary case, 
capitulation phenomena of ``minus class groups'' do not exist and the two 
definitions of the $p$-adic isotopic components coincide).

\smallskip
Finally, in the Iwasawa theory context, non semi-simple by nature, a specific Main 
Conjecture has been proved, in terms of $p$-adic $L$-functions, and is also at the origin of 
many articles (to avoid considerable bibliography, we refer to that given in \cite{Gra2021}; 
a main overview on the story, precise proofs and classical references are given in 
Washington's book \cite[Chapters 6, 8, 13, 15, Notes]{Was1997}). The 
fundamental difference, regarding finite $p$-extensions, is that capitulation 
kernels are hidden in statements using pseudo-isomorphisms
of $\Z_p[[T]]$-modules, whence only giving results for 
the projective limit of the $p$-class groups in the $\Z_p$-extensions and, in 
general, no precise information is available in the finite layers (it's quite clear
in a numerical setting that any possible structure occurs in the first layers, up
to the algebraic regularity predicted by Iwasawa's theory; see for instance the
numerical computations given in \cite{KS1995,Paga2022}).
Nevertheless, in the real case, Greenberg's conjecture \cite{Gree1976}
(saying that $\lambda = \mu = 0$) makes it somewhat unnecessary.
Indeed, if $k$ is totally real and $K = k_\infty = \bigcup_{n \geq 0} k_n$ the 
cyclotomic $\Z_p$-extension of $k$, then $\lambda = \mu = 0$ is equivalent to the
``stability'' of the class groups $\CH_{k_n}$ in the tower, from some layer $n_0 \gg 0$, 
in which case, the arithmetic norms $\Norm_{k_{n+1}/k_n} : \CH_{k_{n+1}} \to \CH_{k_n}$ 
are isomorphisms, the transfer maps $\BJ_{k_\infty/k_n} : 
\CH_{k_n} \to \CH_{k_\infty}$ being never injective since any $\CH_{k_n}$ capitulates 
in $k_\infty$ (see, e.g., \cite[Theorem 1.2, Remark 1.3]{Gra2023} and Grandet--Jaulent 
\cite[Th\'eor\`eme, p. 214]{GrJa1985} proving that $\CH_{k_n} \simeq 
\Ker(\BJ_{{k_\infty}/k_n}) \bigoplus \BJ_{{k_\infty}/k_n}(\CH_{k_n})$ for $n \gg 0$). 

\smallskip
Thus, in the first layers, the $\order \CH_{k_n}$'s are increasing in some random way 
and the RAMC makes sense in these layers, up to the stabilization from 
$k_{n_0}$, where, for all $p$-adic character $\varphi_n$ of $k_n$, $\varphi_n = 
\varphi_{n,0} \varphi_{n,p}$, the relation $\order \CH^\ar_{k_n,\varphi_n} = 
\order(\CE_{k_n} / \CE^\circ_{k_n}\, \CF_{\!k_n})_{\varphi_{n,0}}$ holds 
for all $n \geq 0$ and becomes constant from $n_0$, if RAMC is true.
As we have explained, \textit {the case of even $p$-adic characters in a non 
semi-simple context}, was less understood because of a problematic definition of the 
isotopic components and also of cyclotomic units; the more conceptual case of Iwasawa's 
theory has been privileged and generalized in many directions in the framework of Euler's systems.

\section{Abelian extensions and characters}
Let $\CG$ be the Galois group of the maximal abelian extension $\Q^\ab$ of $\Q$ and
let $K$ be any subfield of finite degree of $\Q^\ab$; put $\CG_K = \Gal(K/\Q)$.

\subsection{Abelian characters}
Let $\Psi$ be the set of irreducible characters of $\CG$, of degree $1$ and finite order, 
with values in an algebraic closure $\ov \Q_p$. We define the set of irreducible $p$-adic characters 
$\Phi$, for a prime $p \geq 2$, the set $\CX$ of irreducible rational characters
and the subsets of irreducible characters $\Psi_K$, $\Phi_K$, $\CX_K$, of $K$.
The notation $\psi \mid \varphi \mid \chi$ (for $\psi \in \Psi$, $\varphi \in \Phi$,
$\chi \in \CX$) means that $\varphi$ is a term of $\chi$ and $\psi$ a term of $\varphi$.
The set $\CX$ has the following elementary property to be considered as the 
``Main Theorem'' for rational components \cite{Leo1954,Leo1962}: 

\begin{theorem} \label{chiformula}
Let $K/\Q$ be an abelian extension and let 
$(A_\chi)_{\chi \in \CX_K}$, $(A'_\chi)_{\chi \in \CX_K}$
be two families of positive numbers, indexed by the set $\CX_K$ 
of irreducible rational characters of $K$. If for all 
subfields $k$ of $K$, one has
$\prod_{\chi \in \CX_k} A'_\chi = \prod_{\chi \in \CX_k} A_\chi$,
then $A'_\chi = A_\chi$ for all $\chi \in \CX_K$.
\end{theorem}

The interest of this property is that analytic formulas (giving for instance 
orders $A_K$ of some finite $p$-adic invariants $\CA_K$ of abelian fields $K$) 
may be {\it canonically} decomposed under the form $A_K = \prod_{\chi \in \CX_K} A_\chi$,
to be compared with algebraic formula $\order \CA_K = \prod_{\chi \in \CX_K} 
\order \CA_\chi$ for suitable canonical $\Z_p[\CG_K]$-modules $\CA_\chi$, so that 
$\order \CA_\chi = A_\chi$ for all $\chi$; the RAMC being the same statement, 
replacing rational characters $\chi$ by $p$-adic ones $\varphi$, under the existence of natural 
relations $\order \CA_\chi = \prod_{\varphi \mid \chi} \order \CA_\varphi$
and $A_\chi = \prod_{\varphi \mid \chi} A_\varphi$.
Unfortunately, classical algebraic definitions of the $\CA_\varphi$'s does not
fulfill these equalities in the non semi-simple case $\order \CG_K \equiv 0 \pmod p$
as we have shown in \cite{Gra2022,Gra2023}.

\smallskip
Noting that the fixed field $K_\chi$ of $\Ker(\chi)$ is cyclic, there is no restriction 
to assume $K/\Q$ real cyclic in all the sequel, then such that $K = K_\chi$.

\subsection{Algebraic versus Arithmetic isotopic components}
Let $K/\Q$ be a real cyclic extension of Galois group $\CG_K \simeq 
\gamma \oplus \Gamma$, $\gamma$ of prime-to-$p$ order $d$,
$\Gamma \simeq \Z/p^e \Z$, $e \geq 1$. Let $k := K^{\Gamma}$ and $K_0 := K^\gamma$.

\subsubsection{Characters of \texorpdfstring{$K$}{Lg}}\label{characters}
Let $\chi$ be the rational character defining $K$. The field of values of $\psi \mid \chi$ 
is $\Q(\mu_{dp^e})$, direct compositum $\Q(\mu_d) \Q(\mu_{p^e})$; thus 
$\psi = \psi_0 \cdot \psi_p$, $\psi_0$ of order $d$,  $\psi_p$ of order $p^e$ and
$\chi = \chi_0 \cdot \chi_p$, $\chi_0 \in \CX_k$ above $\psi_0$, $\chi_p \in \CX_{K_0}$ 
above $\psi_p$. Similarly, in the direct compositum $\Q_p(\mu_d) \Q_p(\mu_{p^e})$,
irreducible $p$-adic characters $\varphi \mid \chi$ are of the form $\varphi_0 \cdot \varphi_p$, 
$\varphi_p = \chi_p$ since $\Gal(\Q_p(\mu_{p^e})/\Q_p)\simeq \Gal(\Q(\mu_{p^e})/\Q)$.

\smallskip
Whence $\Psi_K = \Psi_k \!\cdot\! \Psi_{K_0}$, 
$\Phi_K = \Phi_k \!\cdot\!  \Phi_{K_0}$, $\CX_K = \CX_k \!\cdot\!  \CX_{K_0}$. In the 
writing $\varphi = \varphi_0 \!\cdot\!  \varphi_p$, $\varphi_0$ will be called the 
tame part of~$\varphi$.
Any $\rho \in \CX_K$ corresponds to $K_\rho$, cyclic of degree $d' p^{e'}$, 
$d' \mid d$ and $e' \leq e$; so, any result for $\chi$ and $\varphi \mid \chi$, 
holds in the same way for $\rho$ and its $p$-adic characters. Thus, in general, 
we will state results only for $\chi$.

\subsubsection{Definitions of the isotopic components}
The non semi-simple context is problematic for the definition of isotopic 
$p$-adic components in the form $\CH_{K,\varphi}$, $\varphi \mid \chi$, since
$\CH_K =  \oplus_{\rho \in \CX_K} \oplus_{\varphi \mid \rho}\CH_{K,\varphi}$ 
does not make sense in general.
Let's recall the definitions and explain how the phenomenon of capitulation 
gives rise to difficulties about the classical algebraic definition.

\smallskip
For $\CG_K = \gamma \oplus \Gamma$ cyclic of order $d p^e$, $e \geq 1$, 
and for all $\varphi \mid \chi$, the first classical definition is the following, 
where $P_\chi$ is the $d p^e$th cyclotomic polynomial, $\sigma_\chi$ a 
generator of $\CG_K$ and $P_\varphi \mid P_\chi$ the local cyclotomic 
polynomial associated to the action $\tau \in \CG_K \mapsto \psi (\tau)$,
$\psi \mid \varphi$ and $P_\varphi (\psi (\tau)) = 0$:
\begin{equation}\label{hat}
\widehat \CH^\alg_{K,\varphi} := \CH_K \ \hbox {$\bigotimes^{}_{\Z_p[\CG_K]}$} \ 
\Z_p[\mu_{d p^e}^{}] \simeq \CH_K/P_\varphi (\sigma_\chi) \!\cdot\! \CH_K,
\end{equation}
as $\Z_p[\mu_{d p^e}^{}]$-module (see, e.g., Solomon \cite[II, \S\,1]{Sol1990}, 
Greither \cite[Definition, p. 451]{Grei1992} or Mazigh \cite[Introduction]{Mazigh2017}). 
Another algebraic definition, giving $\Z_p[\mu_{d p^e}^{}]$-modules, is that of
the kernels of the actions of $P_\chi (\sigma_\chi)$ and $P_\varphi (\sigma_\chi)$;
it defines what we will call the notions of {\it algebraic} $\chi$ and $\varphi$-objects:
\begin{equation}\label{normal}
 \hspace{0.55cm}\left \{\begin{aligned}
\CH^\alg_{K,\chi} :=  & \ \{x \in \CH_K, \ P_\chi (\sigma_\chi)  \!\cdot\! x = 1 \}  \\
\CH^\alg_{K,\varphi} := &\ \{x \in \CH_K, \ P_\varphi (\sigma_\chi)  \!\cdot\! x = 1\}.
\end{aligned}\right.
\end{equation}

From such an algebraic definition of the $\chi$-objects $\CH^\alg_{K,\chi}$ in \eqref{normal}, 
we have proved in \cite[\S\,3.2.4, Theorem 3.7, Definition 3.11]{Gra2021}, the  
interpretation:
$$\CH^\alg_{K,\chi} = \{x \in \CH_K, \ \Nu_{\!K/\Kappa}(x) = 1,\, 
\forall \, \BKappa \varsubsetneqq K \}, $$

\noindent
where $\Nu_{\!K/\Kappa} = \sum_{\sigma \in \Gal(K/\Kappa)} \sigma$ is the algebraic 
norm in $K/\Kappa$.

\smallskip
This gives rise to our {\it arithmetic} definitions of $\chi$ and $\varphi$-objects, 
replacing $\Nu_{\!K/\Kappa}$ by the arithmetic norm $\Norm_{K/\Kappa}$, 
noting that $\Nu_{\!K/\Kappa} = \BJ_{\!K/\Kappa} \circ \Norm_{K/\Kappa}$, 
where $\BJ_{\!K/\Kappa}$ is the transfer map (or extension of classes
deduced of that of ideals):
\begin{equation}\label{maindef}
\left \{\begin{aligned}
\CH^\ar_{K,\chi} := &\  \{x \in \CH_K,\ \,\Norm_{K/\Kappa}(x) = 1,\, 
\forall \, \BKappa \varsubsetneqq K \}, \\
\CH^\ar_{K,\varphi} := & \  \{x \in \CH_K, \ \, \Norm_{K/\Kappa}(x) = 1,\, 
\forall \, \BKappa \varsubsetneqq K
\, \ \& \ \, P_\varphi (\sigma_\chi)  \!\cdot\! x = 1 \} 
\end{aligned}\right.
\end{equation}

The above definitions of $\chi$-objects lead to an unexpected semi-simplicity, in accordance 
with analytic formulas and the non semi-simple RAMC \cite[Theorem 4.5]{Gra2021}:

\begin{theorem}\label{isotopicphi}
Let $K/\Q$, $K = K_\chi$, $\chi =: \chi_0 \chi_p$, be a real cyclic extension of Galois 
group $\CG_K \simeq \gamma \oplus \Gamma$, with $\gamma$ of prime-to-$p$ order $d$,
and $\Gamma \simeq \Z/p^e \Z$, $e \geq 1$.
Let $\varphi \mid \chi$, $\varphi =: \varphi_0\cdot \varphi_p$, $\varphi_0 \mid \chi_0$ 
(see \S\,\ref{characters}) and put $e_{\varphi_0} := \frac{1}{d} 
\sum_{\tau \in \gamma} \varphi_0(\tau^{-1}) \, \tau$.
If the $\Z_p[\CG_K]$-module $\CA^\ar_{K,\chi}$ is a $\chi$-object, then 
$\CA^\ar_{K,\chi} = \bigoplus_{\varphi \mid \chi} \CA^\ar_{K,\varphi}$, where 
$\CA^\ar_{K,\varphi} := (\CA^\ar_{K,\chi})^{e_{\varphi_0}}$, 
also denoted $(\CA^\ar_{K,\chi})_{\varphi_0}$ (cf. definitions \eqref{maindef}).
\end{theorem}

For $\CA_K = \CH_K$, since $K/k$ is totally ramified, arithmetic norms $\Norm_{L/K}$ 
are surjective (class field theory); so we have the following relations (not fulfilled by the 
$\CH_{K_\rho,\rho}^\alg$'s because of possible capitulations leading to non surjective 
$\Nu_{K/\Kappa}$'s; see \S\,\ref{capitulations} for some numerical examples):

\begin{corollary}\label{relation}
We have $\order \CH_K = \prd_{\rho \in \CX_K} \order \CH^\ar_{K_\rho,\rho} =
\prd_{\rho \in \CX_K}\prd_{\varphi \mid \rho}\order \CH^\ar_{K_\rho,\varphi}$.
\end{corollary}

\subsubsection{Practical characterization of the 
\texorpdfstring{$\CH^\ar_{K,\varphi}$}{Lg}'s} \label{remafond}
The computation of:
$$\CH^\ar_{K,\varphi} := \  \{x \in \CH_K, \ \, \Norm_{K/\Kappa}(x) = 1,\, 
\forall \, \BKappa \varsubsetneqq K \ \ \& \ \ P_\varphi (\sigma_\chi)  \!\cdot\! x = 1 \}, $$
may be obtained in two commutable steps, from Theorem \ref{isotopicphi}:

\smallskip\noindent
(i) {\bf Computation of $\CH^\ar_{K,\chi}$.}
So, $\CH^\ar_{K,\chi} = \{x \in \CH_K,\, \Norm_{K/\Kappa}(x) = 1,\, \forall \, 
\BKappa \varsubsetneqq K \}$ is the group of ``totally relative classes''.
Let $\ov K$ be the unique subfield of $K$ such that $[K : \ov K] = p$; then
the conditions $\Norm_{K/\Kappa'}(x) = 1$ for the subfields $\BKappa'$, such that 
$K^\gamma \subseteq \BKappa' \varsubsetneqq K$, is automatically fulfilled for the 
$\chi_0$-component of the kernel, $\CH^\ar_{K/\ov K}$, of 
$\Norm_{K/\ov K} : \CH_K \to \CH_{\ov K}$, thus giving for $\chi = \chi_0 \chi_p$:
\begin{equation}\label{relative}
\left \{\begin{aligned}
\CH^\ar_{K/\ov K} & := \{x \in \CH_K,\, \Norm_{K/\ov K}(x) = 1\}, \\
\CH^\ar_{K,\chi} & := (\CH^\ar_{K/\ov K})^{e_{\chi_0}} =: \CH^\ar_{K/\ov K,\chi_0},
\end{aligned}\right.
\end{equation}
where $e_{\chi_0} \in \Z_p[\gamma]$ is the usual semi-simple idempotent.

\smallskip\noindent
(ii) {\bf Computation of $\CH^\ar_{K,\varphi}$.}
From Theorem  \eqref{isotopicphi} and the semi-simple decomposition 
$e_{\chi_0} = \sum_{\varphi_0 \mid \chi_0} e_{\varphi_0}$, one gets, for
$\varphi = \varphi_0 \varphi_p$, $\varphi_0 \mid \chi_0$:
\begin{equation}\label{relativephi}
\left \{\begin{aligned}
\CH^\ar_{K/\ov K} & := \{x \in \CH_K,\, \Norm_{K/\ov K}(x) = 1\}, \\ 
\CH^\ar_{K,\varphi} & =  (\CH^\ar_{K/\ov K})^{e_{\varphi_0}} =:
\CH^\ar_{K/\ov K,\varphi_0}.
\end{aligned}\right.
\end{equation}

It is important to note that $\CH^\ar_{K,\varphi} = \{x \in \CH_{K,\varphi_0},\ 
\Norm_{K/\ov K}(x) = 1 \}$.

\smallskip\noindent
(iii) {\bf Analytic objects}.
The $\Z_p[\CG_K]$-module $\CE_K /\CE^\circ_K \, \CF_{\!K}$ is a $\chi$-object 
since $\Norm_{K/\Kappa}(\CE_K) \subseteq \CE^\circ_K$, for all 
$\BKappa \varsubsetneqq K$. Whence, from (ii):
$$(\CE_K /\CE^\circ_K \, \CF_{\!K})_\varphi = 
(\CE_K /\CE^\circ_K \, \CF_{\!K})^{e_{\varphi_0}} =:
(\CE_K /\CE^\circ_K \, \CF_{\!K})_{\varphi_0}. $$

\section{Class groups of real abelian fields -- Analytic formulas}\label{secIII}
Denote by $\BE_K := \vert E_K \vert$ the group of absolute value of the units of $K$, the 
Galois action being defined by $\vert \varepsilon \vert^\sigma = \vert \varepsilon^\sigma \vert$ 
for any unit $\varepsilon$ and any $\sigma \in \CG_K$.

\subsection{The Leopoldt cyclotomic units}\label{subIII2} 
This aspect being very classical, we just recall the definitions
(see Leopoldt \cite[\S\,8\ (1)]{Leo1954}, \cite{Leo1962}, or Washington \cite[Chap. 8]{Was1997}).

\begin{definition}\label{defIII3} 
(i) Let $\rho \in \CX_K \setminus \{1\}$ be of conductor $f_\rho$, let
$\zeta_{2f_\rho} := \exp \Big(\ffrac{i \pi}{f_\rho} \Big)$, and put
$\theta_\rho := \prod_{a \in A_\rho} (\zeta_{2f_\rho}^a - \zeta_{2f_\rho}^{-a})$,  
where $A_\rho \subset (\Z/ f_\rho \Z)^\times$ is a half-system of repre\-sentatives 
corresponding to $\Gal \big(\Q(\zeta_{2f_\rho} + \zeta_{2f_\rho}^{-1})/K_\rho \big)$. 

\smallskip
(ii) Let $\BF_K$ be the intersection with $\BE_K$ of the multiplicative group generated 
by the conjugates of the $\vert \theta_\rho \vert$'s, for $\rho \in \CX_K \setminus \{1\}$. 
This defines the group of cyclotomic units of $K$ and we put 
$\CE_K := \BE_K \otimes \Z_p$, $\CF_{\!K} := \BF_K \otimes \Z_p$. 

\smallskip
(iii) For the cyclic real field $K$ of degree $d p^e$, let's denote by $\BE^\circ_K$ the 
subgroup of $\BE_K$ generated by the $\BE_{\Kappa}$'s for all the subfields 
$\BKappa \varsubsetneqq K$.
\end{definition}

\subsection{Arithmetic computation of 
\texorpdfstring{$\order \CH^\ar_{K,\chi}$}{Lg}}\label{subIII3}

Since the RAMC is trivial for characters $\chi$ of $p$-power order ($\varphi = \chi$), 
we assume $d = [k : \Q] >1$ and $\chi = \chi_0 \chi_p$, $\chi_0 \ne 1$. 
In that case, interpretation of Leopoldt's formulas \cite{Leo1954,Leo1962} yields, 
in the spirit of Theorem \ref{chiformula} (see \cite[pp. 71--75]{Gra1976} for  
details and proofs, or \cite[Section 7, \S\S\,7.1--7.3]{Gra2021} for an overview):
\begin{theorem}\label{chiformulaH}
Let $\CH^\ar_{K,\chi} := \{x \in \CH_K,\,  
\Norm_{K/\Kappa}(x) = 1,\ \hbox{for all $\BKappa \varsubsetneqq K$} \}$. Then
$\order \CH^\ar_{K,\chi} =  \big (\CE_K\! : \CE^\circ_K \, \CF_{\!K} \big)$, where
$\CE^\circ_K$ is the subgroup of $\CE_K$ generated by the $\CE_{\Kappa}$'s for 
all the subfields $\BKappa \varsubsetneqq K$.
\end{theorem}

Our purpose is to obtain $\order \CH^\ar_{K,\varphi} =  
\order (\CE_K / \CE^\circ_K\, \CF_{\!K} )_{\varphi_0}$, for all $\varphi_0 \mid \chi_0$,
which constitutes the RAMC. Such a result will be obtained if there exists an
auxiliary cyclotomic extension $K(\mu_\ell^{})$, $\ell$ inert in $K$, in which $\CH_K$ 
capitulates, a trick showing that some natural properties ``above $K$'' are 
logically overhead the arithmetic of the base field.

\begin{remark}\label{e0} 
{\rm The analytic formula given by this theorem, which appears to have been less known, 
seems more convenient than formulas using other groups $\CF'_K$ of cyclotomic units. 
Indeed, compare with \cite[Theorem 4.14]{Grei1992} using instead {\it algebraic 
isotopic components} in a partial semi-simple statement saying that for all
$\varphi_0 \in \Phi_k$, $\order (\CE_K / \CF'_K)_{\varphi_0} = 
\order \CH_{K,\varphi_0}$, up to an explicit power of $2$ and an index 
$(\Z[\CG_K] : U)$ given in the framework of Sinnott's group $\CF'_K$ of cyclotomic units
larger than classical Leopoldt's group of Definition \ref{defIII3}.}
\end{remark}

\subsection{Norm properties of Leopoldt's cyclotomic units}

 Let $f >1$ and let $m \mid f$, with $m>1$, be any modulus; let 
$\Q(\zeta_m) \subseteq \Q(\zeta_f)$ be the corresponding 
cyclotomic fields with $\zeta_t := \exp \big(\frac{2 i \pi}{t}  \big)$, for all $t \geq 1$.  
Put $\eta^{}_f := 1- \zeta_f$, $\eta^{}_m := 1- \zeta_m$. It is well known
that $\Norm_{\Q(\zeta_f)/\Q(\zeta_m)} (\eta^{}_f) = 
 \eta_m^\Omega$, with $\Omega = {\prod_{\ell \mid f,\ \ell \nmid m}
\big(1-\big(\frac{\Q(\zeta_m)}{\ell} \big)^{-1}\big)}$,
where $\big(\frac{\Q(\zeta_m)}{\ell} \big) \in \Gal(\Q(\zeta_m)/\Q)$ denotes the Frobenius 
(or Artin) automorphism of the prime number $\ell \nmid m$ (see, e.g., \cite[\S\,4.2]{Gra2022}).
Applying this result to our context $L/K/\Q$ yields:

\begin{proposition}\label{cycloformula}\label{invertible}
Let $m$ be the conductor of $K$ and let $L \subset K(\mu_\ell)$, 
$\ell$ being totally 
inert in $K$; the conductor $f$ of $L$ is $\ell m$. Set $\eta^{}_L := 
\Norm_{\Q^f/L}(\eta^{}_f)$ and $\eta^{}_K := \Norm_{\Q^m/K}(\eta^{}_m)$;
then, $\Norm_{L/K} (\eta^{}_L) = \eta_K^\Omega,\ \,\hbox{with \,$\Omega = 
1-\big(\frac{K}{\ell} \big)^{-1}$}$, and $\Omega e_{\varphi_0}$ is invertible in 
$\Z_p[\CG_K]e_{\varphi_0}$, for all $\varphi_0 \in {\bf \Phi}_k \setminus \{1\}$. 
Then $\Norm_{L/K}(\CF_{\!L,\varphi_0}) = \CF_{\!K,\varphi_0}$ with the 
Definition \ref{defIII3} of cyclotomic units.
\end{proposition}

\begin{proof}
Indeed, $\tau^{}_{\ell,K} := \big(\frac{K}{\ell} \big)^{-1}$ generates $\CG_K$
since $\ell$ is inert in $K/\Q$; write $\tau^{}_{\ell,K} =: \tau  \sigma$, $\tau$
of order $d$, $\sigma$ of order $p^e$. So for $\psi_0 \mid \varphi_0 \ne 1$, 
$\psi_0 \big(1-\tau \big) = 1 - \psi_0\big(\tau \big)$ is a unit of 
$\Q(\mu_d^{})$ if $\xi := \psi_0 \big (\tau \big)$ is not of prime power order, 
otherwise, $\big(1 - \xi \big)$ is a prime ideal above a prime $q \mid d$ and 
$1 - \xi$ is a $p$-adic unit. Whence, $\psi_0 \big (1-\tau_{\ell,K}\big) =
1 - \xi +\xi (1-\sigma)$; thus $1-\tau_{\ell,K}$ in invertible in $\Z_p[\CG_K]$ 
since $\psi_p(1-\sigma)$ is an uniformizing parameter at $p$.
Since $\ell$ is totally inert in $K/\Q$, this applies to any subfield $K'$ 
of $K$, of character $\chi'= \chi'_0 \chi'_p$ with $\chi'_0 \ne 1$ and
its $p$-adic characters $\varphi'_0 \ne 1$, with $L' = L \cap K'(\mu_\ell^{})$ 
and $\Omega' = 1-\big(\frac{K'}{\ell} \big)^{-1}$. 

\smallskip
The module $\CF_{\!K,\varphi_0}$, $\varphi_0 \mid \chi_0$, only depends 
on the cyclotomic units of the fields $K'$ of characters $\chi'= \chi_0 \chi_p^{p^j}$, 
$0 \leq j \leq e$, that is to say, $k \subseteq K' \subseteq K$.

\smallskip
Whence the norm relation between the $\varphi_0$-components of the Leopoldt 
cyclotomic units for any $\varphi_0 \ne 1$; indeed, Leopoldt's definition \ref{defIII3}
writes $\theta_{f} = \zeta_{2f} - \zeta_{2f}^{-1} = - \zeta_{2f}^{-1}(1 - \zeta_{2f}^2) =
- \zeta_{2f}^{-1}(1 - \zeta_{f}) = - \zeta_{2f}^{-1} \eta^{}_f$ and norms are taken over 
real fields $L$ and $K$.
\end{proof}

\section{The \texorpdfstring{$p$}{Lg}-localization of the Chevalley--Herbrand formula}

Let $K/\Q$ be a cyclic extension of Galois group $\CG_K = \gamma \oplus \Gamma$, 
$\gamma$ of prime-to-$p$ order $d$ and $\Gamma \simeq \Z/p^e \Z$, $e \geq 1$;
let $k = K^\Gamma$, $K_0 = K^\gamma$. Let $M_0/\Q$ be a real cyclic extension 
of degree $p^n$, $n \geq 1$, linearly disjoint from $K$, and let $L= M_0 K$; so any 
prime ideal of $K$, ramified in $L/K$, is totally ramified (this will be the principle 
for the proof of the RAMC, only using $M_0 \subset \Q(\mu_\ell^{})$, 
$\ell \equiv 1 \pmod {2 p^N}$, $N \geq n$, $\ell$ inert in $K/\Q$, under
capitulation of $\CH_K$ in $L$). 
\unitlength=0.40cm
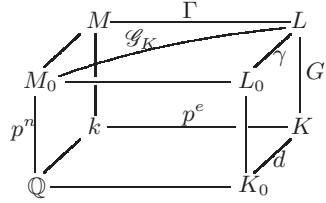
$$\vbox{\hbox{\hspace{-2.0cm} 
\vspace{-0.2cm}
\begin{picture}(11.5,6.4)
\put(5.5,6){\line(1,0){6}}
\put(4.0,4){\line(1,0){5.5}}
\put(5.4,2.50){\line(1,0){4.25}}
\put(10.1,2.50){\line(1,0){1.3}}
\put(3.5,0.50){\line(1,0){6}}
\bezier{350}(3.3,1.)(3.9,1.55)(4.5,2.1)
\bezier{350}(3.3,4.5)(3.9,5.05)(4.5,5.6)
\bezier{350}(10.3,1.)(10.9,1.55)(11.5,2.1)
\bezier{350}(10.3,4.5)(10.9,5.05)(11.5,5.6)
\put(3.0,1.0){\line(0,1){2.5}}
\put(5.0,2.9){\line(0,1){0.98}}
\put(5.0,4.2){\line(0,1){1.4}}
\put(10.0,1.0){\line(0,1){2.5}}
\put(11.8,3.0){\line(0,1){2.5}}
\put(2.75,0.3){\ft$\Q$}
\put(4.75,2.3){\ft$k$}
\put(9.75,0.35){\ft$K_0$}
\put(10.9,1.2){\ft$d$}
\put(11.5,2.3){\ft$K$}
\put(2.2,2.2){\ft$p^n$}
\put(7.9,2.7){\ft$p^e$}
\put(2.6,3.8){\ft$M_0$}
\put(4.7,5.8){\ft$M$}
\put(9.75,3.8){\ft$L_0$}
\put(11.5,5.8){\ft$L$}
\put(7.9,6.15){\ft$\Gamma$}
\put(10.9,4.7){\ft$\gamma$}
\put(12.0,4.1){\ft$G$}
\put(6.0,5.2){\ft$\CG_K$}
\bezier{200}(3.8,4.2)(7.1,5.4)(11.3,5.8)
\end{picture}   }} $$
\unitlength=1.0cm

Under the previous assumptions, the Chevalley--Herbrand formula for $p$-class groups 
\cite[pp. 402-406]{Che1933} is the following, where $G := \Gal(L/K)$:
$$\order \CH_L^G = \ds  \order \CH_K \times \frac{p^{n (r-1)}}
{(E_K : E_K \cap \Norm_{L/K}(L^\times))},$$ 
$r \geq1$ being the number of 
prime ideals of $K$ ramified in $L/K$. If $r=1$, then:
$$\order \CH_L^G = \order \CH_K, $$
which does not imply any group isomorphism $\CH_L^G \simeq \CH_K$;
this subtlety of Chevalley--Herbrand formula will be the key for the proof which
needs to compute the $\order \CH_{L,\varphi_0}^G$'s, 
$\varphi_0 \in \Phi_k$, that we call the $p$-localizations of the formula; even if 
one could guess the result, the frequent capitulation (total or not) of $\CH_K$ in $L$, 
goes against a trivial proof.

\subsection{The \texorpdfstring{$p$}{Lg}-localizations of 
\texorpdfstring{$\CH_L^G$}{Lg} as \texorpdfstring{$\Z_p[\CG_K]$}{Lg}-module}

The $p$-localizations of the Chevalley--Herbrand formula were given, for the case 
$p \nmid [K : \Q]$ (i.e., $K_0=\Q$), in the 1970/1980's (see \cite{Gra2022,Gra2023} 
for some story), but here we must give some improvements, regarding $\chi$ 
as character of $\Gal(L/M_0) \simeq \CG_K$, with $\chi_p \ne 1$. In other words, for 
$\varphi_0 \in \Phi_k$, we consider $\CH_{L,\varphi_0}$ as $\Z_p[\CG_K]$-module for 
which $(\CH_{L,\varphi_0})^G = (\CH_L^G)^{e_{\varphi_0}}$ (denoted $\CH_{L,\varphi_0}^G$), 
since $(\CH_L^{e_{\varphi_0}})^G =  (\CH_L^G)^{e_{\varphi_0}}$ by commutativity
(non semi-simple $\Gamma$-modules).

\smallskip
To get a formula for the orders of the $\varphi_0$-components $\CH_{L,\varphi_0}^G$, 
we follow the process given in Jaulent's Thesis \cite[Chapitre III, p. 167]{Jau1986}.

\begin{theorem}\label{chevalleylocal}
Let $K/\Q$ be a real cyclic extension of Galois group $\CG_K = \gamma \oplus \Gamma$, 
$\gamma$ of prime-to-$p$ order $d$ and $\Gamma \simeq \Z/p^e \Z$, $e \geq 1$.
Let $k := K^\Gamma$.

\smallskip\noindent
Let $M_0/\Q$ be the real cyclic extension of degree $p^n$, $n \in [1, N]$,
contained in $\Q(\mu_\ell^{})$, where $\ell \equiv 1 \pmod{2p^N}$;
we assume that $\ell$ is totally inert in $K$.
Put $L = M_0 K$ and $G := \Gal(L/K) =: \langle \sigma \rangle$.  
Let $\Bj_{L/K}^{}$ be the extension of ideals $I_K \to I_L^G$ and
$\BJ_{L/K}$ the corresponding transfer $\BH_K \to \BH_L^G$ of classes.

\smallskip
{\bf a)} We have the exact sequence of $\Z[\CG_K \times G]$-modules:
\begin{equation*}
\begin{aligned}
1 \too & \Ker(\BJ_{L/K}) \too  \Hom^1(G,E_L) \too \Coker (\Bj_{L/K}^{}) \too  \\
\too & \Coker (\BJ_{L/K}) \too \Hom^2(G, E_L) \too  E_K/E_K \cap \Norm(L^\times) \to 1;
\end{aligned}
\end{equation*}

\noindent
 If $\CH_K$ capitulates in $L$, then:
 
 \smallskip
{\bf b)}  $\order \CH_{L,\varphi_0}^G =  \order \CH_{K,\varphi_0}$,
for all $\varphi_0 \in \Phi_k \setminus \{1\}$;

\smallskip
{\bf c)}  
$\CH_{K,\varphi_0}\! \simeq \Hom^1(G,\CE_{L,\varphi_0})\! =
\CE^*_{L,\varphi_0}/ \CE_{L,\varphi_0}^{1-\sigma}$ and
$\CH_{L,\varphi_0}^G \! \simeq \Hom^2(G,\CE_{L,\varphi_0}) = 
\CE_{K,\varphi_0}/ \Norm_{L/K}(\CE_{L,\varphi_0})$, for all $\varphi_0 \in \Phi_k \setminus \{1\}$. 
\end{theorem}

\noindent
{\bf Proof.}
Let $I_K$ and $I_L$ (resp.  $P_K$ and $P_L$) be the ideal groups
(resp. the subgroups of principal ideals), of $K$ and $L$, respectively. 
Consider the exact sequences of $\Z[\CG_K \times G]$-modules:
\begin{equation}\label{ab}
1 \to E_L \to L^\times \to P_L \to 1, \ \ \ \  1 \to E_K \to K^\times \to P_K \to 1
\end{equation}
\begin{equation}\label{cd}
1 \to P_L \to I_L \to \BH_L \to 1, \ \ \ \ 1 \to P_K \to I_K \to \BH_K \to 1.
\end{equation}

\begin{lemma}\label{lemma0}
We have the following properties:

\smallskip
(i) $\Coker \big(\Bj'_{L/K} : P_K \to P_L^G \big) \simeq \Hom^1(G,E_L)$;

\smallskip
(ii) $\Hom^1(G,P_L) \simeq \Ker \big[  \Hom^2(G,E_L) \!\to \!\Hom^2(G,L^\times)  \big]$;

\smallskip
(iii)  $\Hom^1(G,I_L) = 1$.
\end{lemma}

\begin{proof}
We have, from exact sequences \eqref{ab}:
\begin{equation*}
\begin{aligned}
1 & \to E_L^G = E_K \to L^\times{}^G = K^\times \to P_L^G \to 
\Hom^1(G,E_L) \to  \Hom^1(G,L^\times)  \\
& \to  \Hom^1(G,P_L) \to \Hom^2(G,E_L) \to \Hom^2(G,L^\times).
\end{aligned}
\end{equation*}

Since $\Hom^1(G,L^\times)  = 1$ (Hilbert's Theorem 90), this yields (i) and (ii).
The claim (iii) is classical since $I_L$ is a $\Z[G]$-module generated
by the prime ideals of $L$ on which the Galois action is canonical.
\end{proof}

From exact sequences \eqref{cd}, we obtain the following diagram:
$$\begin{array}{ccccccccc}
1  &  \!\! \too \!\! & P_K  &  \!\! \tooo \!\! & I_K &  \!\! \tooo \!\! & \BH_K & \!\! \too \ \ 1 & \\   
\vspace{-0.3cm}   \\
& &   \Big \downarrow &  
&\hspace{-2.9cm} \hbox{\ft$\Bj'_{L/K}$\ns}
&\hspace{-1.1cm}\Big \downarrow \hbox{\ft$\Bj_{L/K}^{}$\ns}  
&& \hspace{-1.6cm}\Big \downarrow  \hbox{\ft$\BJ_{L/K}$\ns}  \\   
 \vspace{-0.3cm}    \\
1  &   \!\! \too \!\!   & P_L^G  &  \!\! \tooo \!\! &  I_L^G &  \!\! \tooo \!\! &  \BH_L^G & 
\!\! \!\! \!\! \!\! \!\! \!\! \too \!\! & \!\! \!\!  \!\! \!\! \!\Hom^1(G,P_L) \too 1\,.
\end{array} $$

The snake lemma gives the exact sequence:
\begin{equation*}
\begin{aligned}
1& \to \Ker(\Bj'_{L/K}) \to \Ker(\Bj_{L/K}^{}) \too \Ker(\BJ_{L/K})  
\to \Coker(\Bj'_{L/K}) = \Hom^1(G,E_L) \\
& \to \Coker(\Bj_{L/K}^{}) \to  \Coker(\BJ_{L/K})  
\to  \Hom^2(G,E_L) \ds \mathop{\to}^{u\ } \Hom^2(G,L^\times),
\end{aligned}
\end{equation*}

\noindent
since $\Ker(\Bj'_{L/K}) = \Ker(\Bj_{L/K}^{}) = 1$ and
$\Hom^2(G,L^\times) \simeq K^\times/\Norm_{L/K}(L^\times)$, the above
exact sequence becomes with ${\rm Im}(u) = E_K/E_K \cap \Norm_{L/K}(L^\times)$:
\begin{equation*}
\begin{aligned}
& 1 \to \Ker(\BJ_{L/K}) \to \Hom^1(G,E_L)   \to \Coker(\Bj_{L/K}^{}) \to \\
  & \Coker(\BJ_{L/K})  \to \Hom^2(G,E_L) 
\to E_K/E_K \cap \Norm_{L/K}(L^\times) \to 1,
\end{aligned}
\end{equation*}
proving {\bf (a)}.
Whence the localized exact sequences of $\Z_p[\CG_K]$-modules:
\begin{equation}
\left \{\begin{aligned}
&1 \to  \Ker(\BJ_{L/K})_{\varphi_0} \to  \Hom^1(G,\CE_{L,\varphi_0}) \to  
\Coker(\Bj_{L/K}^{})_{\varphi_0}  \to \\
&\Coker(\BJ_{L/K})_{\varphi_0}  \to  \Hom^2(G,\CE_{L,\varphi_0}) \to 
\CE_{K,\varphi_0}/\CE_{K,\varphi_0} \cap \Norm_{L/K}(L^\times) \to 1;
\end{aligned}\right.
\end{equation}

\noindent
since  $L/K$ is totally ramified at the unique prime $(\ell)$ of $K$
(inert in $K/\Q$),
$\Coker (\Bj_{L/K}^{})_{\varphi_0} = 1$, 
$(\CE_K/\CE_K \cap \Norm_{L/K}(L^\times))_{\varphi_0} = 1$
for $\varphi_0 \in \Phi_k \setminus \{1\}$, giving:
\begin{equation}\label{kercoker}
\left \{\begin{aligned}
&1 \too  \Ker(\BJ_{L/K})_{\varphi_0} \too  \Hom^1(G,\CE_{L,\varphi_0}) \too 1  \\
&1 \too  \Coker(\BJ_{L/K})_{\varphi_0} \too  \Hom^2(G,\CE_{L,\varphi_0}) \too 1.
\end{aligned}\right.
\end{equation}

\noindent
The capitulation of $\CH_{K,\varphi_0}$ in $L$ yields the semi-simple $p$-localized formula:
\begin{equation}\label{plocal}
\frac{\order \Coker(\BJ_{L/K})_{\varphi_0}}{\order \Ker(\BJ_{L/K})_{\varphi_0}} = 
\frac{\order \CH_{L,\varphi_0}^G}{\order \CH_{K,\varphi_0}} = 
\frac{\order \Hom^2(G,\CE_{L,\varphi_0})}{\order \Hom^1(G,\CE_{L,\varphi_0})},
\end{equation}

\noindent
where $h(\CE_{L,\varphi_0}) := \ds \frac{\order \Hom^2(G,\CE_{L,\varphi_0})}
{\order \Hom^1(G,\CE_{L,\varphi_0})}$
is the Herbrand quotient of $\CE_{L,\varphi_0}$. 
Since $(E_L \otimes\, \Q)\, \oplus\, \Q$ is the regular representation,
there exists a ``Minkowski unit'' $\varepsilon$ such that the 
$\Z[\gamma][\Gamma \times G]$-module
generated by $\varepsilon$ is of prime-to-$p$ index in $E_L$; so,
$\CE_L \oplus \Z_p \simeq \Z_p[\gamma][\Gamma \times G]$ as $\gamma$-modules. 
Thus, $(\CE_L \oplus \Z_p)_{\varphi_0} = \CE_{L,\varphi_0} \simeq
\Z_p[\mu_{d'}] [\Gamma \times G]$ for $\varphi_0 \ne 1$, proving that $h(\CE_{L,\varphi_0})=1$. 
In particular, this yields $\order \CH_{L,\varphi_0}^G =  \order \CH_{K,\varphi_0}$ (point {\bf (b)});
then {\bf (c)} comes from \eqref{kercoker} giving the isomorphisms $\CH_{L,\varphi_0}^G \simeq 
\CE_{K,\varphi_0}/ \Norm_{L/K}(\CE_{L,\varphi_0})$ and $\CH_{K,\varphi_0} \simeq
\CE^*_{L,\varphi_0}/ \CE_{L,\varphi_0}^{1-\sigma}$.
\qed

\subsection{Proof of the RAMC under capitulation in 
\texorpdfstring{$K(\mu_\ell^{})$}{Lg}}\label{fin}
We can think about the main result involving Theorem \ref{chevalleylocal}\,{\bf (b),\,(c)},
under capitulation of $\CH_K$ in $L \subset K(\mu_\ell^{})$, $\ell  \equiv 1\!\! \pmod {2 p^N}$
totally inert in $K$; so:
\begin{equation}\label{f0}
\order \CH_{L,\varphi_0}^G =  \order \CH_{K,\varphi_0}\ \ \, \& \ \ \,  
\CH_{L,\varphi_0}^G \simeq (\CE_K / \Norm_{L/K}(\CE_L))_{\varphi_0},\ 
\varphi_0 \in \Phi_k \setminus \{1\}.
\end{equation}

Knowing only the relation 
$\order \CH^\ar_{K,\chi} = (\CE_K : \CE^\circ_K\, \CF_{\!K})$ (Theorem \ref{chiformulaH}), 
we have to prove (Theorem \ref{isotopicphi} giving the semi-simple decomposition 
of an arithmetic $\chi$-object into its $\varphi$-components), that:
$$\order \CH_{K,\varphi}^\ar = \order (\CE_K / \CE^\circ_K \, \CF_{\!K})_{\varphi_0},
\ \forall \, \varphi = \varphi_0 \chi_p,\, \  \varphi_0 \mid \chi_0 ; $$ 
in fact, we will prove $\order \CH_{K,\varphi}^\ar \leq
\order (\CE_K / \CE^\circ_K \, \CF_{\!K})_{\varphi_0}$, but the above global 
formula $\order \CH^\ar_{K,\chi} = \order (\CE_K / \CE^\circ_K \, \CF_{\!K})$ in which: 
$$\hbox{
$\order \CH^\ar_{K,\chi} = \prod_{\varphi \mid \chi}\order \CH^\ar_{K,\varphi}$
and $\order (\CE_K / \CE^\circ_K \, \CF_{\!K}) = \prod_{\varphi_0 \mid \chi_0}
\order (\CE_K / \CE^\circ_K \, \CF_{\!K})_{\varphi_0}$,} $$
will imply equalities. 

\begin{lemma} \label{cap0}Let $\ov K$ and $\ov L$ be the subfields 
of $K$ and $L$, respectively, such that $[K : \ov K] = [L : \ov L]  = p$.
If $\CH_K$ capitulates in $L$, then $\CH_{\ov K}$ capitulates in $\ov L$.
\end{lemma}

\begin{proof}
Recall that $\Norm_{L/K}(\CH_L) = \CH_K$ and $\Norm_{L/\ov L}(\CH_L) = \CH_{\ov L}$
since the extensions are totally ramified. The capitulation of $\CH_K$ in $L$ is
equivalent to $\Nu_{L/K}(\CH_L) = 1$ since $\Nu_{L/K} = \BJ_{L/K} \circ \Norm_{L/K}$.
If so, $1 = \Norm_{L/\ov L} \circ \Nu_{L/K}(\CH_L) = \Nu_{\ov L/\ov K}(\Norm_{L/\ov L}(\CH_L))
= \Nu_{\ov L/\ov K}(\CH_{\ov L})$ and $\CH_{\ov L}$ capitulates in $\ov L$ (analogous results 
for the subfields $k_i$ of $K$ regarding the extensions $k_i L^\Gamma \subset L$).
\end{proof}

\begin{theorem}\label{relfond}
Let $K/\Q$ be a real cyclic extension of Galois group $\CG_K = \gamma \oplus \Gamma$, 
$\gamma$ of prime-to-$p$ order $d$ and $\Gamma \simeq \Z/p^e \Z$, $e \geq 1$.
Let $M_0/\Q$ be the real cyclic extension of degree $p^N$,
contained in $\Q(\mu_\ell^{})$, where $\ell \equiv 1 \pmod{2p^N}$, $N \geq 1$,
and let $L = M_0K$; we assume, moreover, that $\ell$ is totally inert in $K/\Q$. 

\noindent
If $\CH_K$ capitulates in $L$, then
$\order \CH^\ar_{K,\varphi} = \order (\CE_K / \CE^\circ_K \CF_{\!K})_{\varphi_0}$,
for all $\varphi \mid \chi$, giving the RAMC for $K$.
\end{theorem}

\begin{proof}
For short, put $\Norm = \Norm_{L/K} = \Norm_{\ov L/\ov K}$, 
$\ov \Norm  = \Norm_{L/{\ov L}} = \Norm_{K/{\ov K}}$.
An index $(\CA : \CB)_{\varphi_0}$ means $\order(\CA / \CB)_{\varphi_0} = 
\order (\CA_{\varphi_0} / \CB_{\varphi_0})$. We compute $\CH^\ar_{K,\varphi}$ 
as the set of elements of $\CH_{K,\varphi_0}$ such that $\ov \Norm (x) = 1$
(i.e., $\CH^\ar_{K/\ov K,\varphi_0}$).

We have $(\CE_K :  \CE^\circ_K \CF_{\!K})_{\varphi_0} =
(\CE_K : \CE_{\ov K} \CF_{\!K})_{\varphi_0}$ since for the fields $\BKappa$
 such that $K^\gamma \subseteq \BKappa \varsubsetneqq K$, $\CE_{{\Kappa},\varphi_0}=1$.
We have the exact sequence (cf. \eqref{relative}, \eqref{relativephi})
$1 \to \CH^\ar_{K/\ov K,\varphi_0}\! =  \CH^\ar_{K,\varphi} \to \CH_{K,\varphi_0} 
\ds \mathop{\to}^{\ov \Norm} \CH_{\ov K,\varphi_0} \to 1$,
then, from \eqref{f0} and Lemma \ref{cap0}, $\order \CH_{K,\varphi_0}  
= (\CE_K : \Norm (\CE_L))_{\varphi_0}$ and
$\order \CH_{\ov K,\varphi_0} = (\CE_{\ov K} : \Norm (\CE_{\ov L}))_{\varphi_0}$,
giving, under capitulations, the fundamental relation:
\begin{equation}\label{f6}
\order \CH^\ar_{K,\varphi} = \frac{( \CE_K : \Norm (\CE_L))_{\varphi_0}}
{(\CE_{\ov K} : \Norm (\CE_{\ov L}))_{\varphi_0}},
\end{equation}
to be compared with $X_{\varphi_0} := \big(\CE_K : \CE_{\ov K} \CF_{\!K}\big)_{\varphi_0}$.
From the exact sequence:
\begin{equation*}
1 \to (\CE_{\ov K}\Norm (\CE_L)/\Norm (\CE_L))_{\varphi_0}
\too  (\CE_K/\Norm (\CE_L))_{\varphi_0} \too 
( \CE_K/\CE_{\ov K}\Norm (\CE_L))_{\varphi_0} \to 1
\end{equation*}
and Corollary \ref{invertible} giving $\CE_{\ov K} \CF_{\!K} = \CE_{\ov K} \Norm (\CF_{\!L}) 
\subseteq \CE_{\ov K}\Norm (\CE_L) \subseteq \CE_K$, one gets:
\begin{equation}\label{f7}
\begin{aligned}
X_{\varphi_0} = \frac{(\CE_K : \Norm (\CE_L))_{\varphi_0}}
{(\CE_{\ov K} \Norm (\CE_L) : \Norm (\CE_L))_{\varphi_0}} \times 
\big(\CE_{\ov K} \Norm (\CE_L) : \CE_{\ov K} \CF_{\!K} \big)_{\varphi_0}.
\end{aligned}
\end{equation}
\noindent
Let:
\begin{equation}\label{f8}
Y_{\varphi_0} := \big(\CE_{\ov K} \Norm (\CE_L) : \CE_{\ov K} \CF_{\!K} \big)_{\varphi_0}\ \ 
{\rm and}\ \  \ov \CF_{\!K} :=  \CF_{\!K} \cap  \CE_{\ov K};
\end{equation}
since 
$\CE_{\ov K} \cap \big[ \Norm (\CE_{\ov L}) \CF_{\!K}\big] 
= \Norm(\CE_{\ov L})\ov \CF_{\!K}$, we have the exact sequence: 
\begin{equation*}
\begin{aligned}
1 \to (\Norm (\CE_{\ov L})\ov \CF_{\!K} / \Norm(\CE_{\ov L}))_{\varphi_0} 
& \too (\CE_{\ov K}/\Norm (\CE_{\ov L}))_{\varphi_0} \\
& \too (\CE_{\ov K}  \CF_{\!K} / \Norm (\CE_{\ov L})\CF_{\!K})_{\varphi_0} \to 1,
\end{aligned}
\end{equation*}
and $\Norm (\CE_{\ov L})\CF_{\!K} \subseteq \CE_{\ov K} \CF_{\!K} 
\subseteq \CE_{\ov K}  \Norm (\CE_L)$, then giving from \eqref{f8}:
\begin{equation*}
\begin{aligned}
Y_{\varphi_0} =& \frac{(\CE_{\ov K} \Norm (\CE_L) :  
\Norm (\CE_{\ov L})\CF_{\!K})_{\varphi_0}}{(\CE_{\ov K}  \CF_{\!K} :
\Norm (\CE_{\ov L}) \CF_{\!K})_{\varphi_0}} \\
= & \frac{(\CE_{\ov K} \Norm (\CE_L) :
\Norm (\CE_{\ov L}) \CF_{\!K})_{\varphi_0} \times 
(\Norm (\CE_{\ov L}) \ov \CF_{\!K} : \Norm (\CE_{\ov L}))_{\varphi_0}}
{(\CE_{\ov K} : \Norm (\CE_{\ov L}))_{\varphi_0}}.
\end{aligned}
\end{equation*}

\noindent
So, from \eqref{f6}, \eqref{f7}, $X_{\varphi_0} = \ds
\frac{(\CE_K : \Norm (\CE_L))_{\varphi_0}}
{(\CE_{\ov K} : \Norm (\CE_{\ov L}))_{\varphi_0}} \times Z_{\varphi_0}
= \order \CH^\ar_{K,\varphi}\! \times Z_{\varphi_0}$, where:
\begin{equation*}
Z_{\varphi_0} :=  \frac{(\CE_{\ov K} \Norm (\CE_L) : \Norm (\CE_{\ov L}) \CF_{\!K})_{\varphi_0}}
{(\CE_{\ov K} \Norm (\CE_L) : \Norm (\CE_L))_{\varphi_0}}
\times (\Norm (\CE_{\ov L}) \ov \CF_{\!K} : \Norm (\CE_{\ov L}))_{\varphi_0} \geq 1,
\end{equation*}

\noindent
since $\Norm (\CE_{\ov L}) \CF_{\!K} \subseteq \Norm (\CE_L)$.
Thus, inequalities $\order \CH^\ar_{K,\varphi} \leq (\CE_K : \CE_{\ov K} \CF_{\!K})_{\varphi_0}$ 
are equalities because of product formulas over $\varphi \mid \chi$
giving, as explained, $\order \CH^\ar_{K,\chi} = (\CE_K\! : \CE^\circ_K \, \CF_{\!K})$;
whence the theorem and, since $Z_{\varphi_0} = 1$, the supplementary information:
$$\Norm (\CE_{L,\varphi_0}) = \Norm (\CE_{\ov L,\varphi_0}) \CF_{\!K,\varphi_0} =
\Norm (\CE_{\ov L,\varphi_0} \CF_{\!L,\varphi_0}), $$ 
which generalizes the semi-simple case,
$\Norm (\CE_{L,\varphi}) = \Norm ( \CF_{\!L,\varphi})$ \cite[Theorem 4.6]{Gra2022},
and justifies that relation \eqref{f6} does not seem to depend on cyclotomic units, which 
may look mysterious but comes from the specific properties of $L/K$.
\end{proof}

\subsection{Numerical illustrations of capitulation phenomenon}\label{capitulations}
Let's consider, for $p=3$, $k = \Q(\sqrt{229})$ (class group of order $3$)
and the cubic field $K_0$ of conductor $37$ (trivial $3$-class group);
PARI \cite{Pari2016} gives for $K = k K_0$ the class group ${\sf [3,3,3]}$.
For $\ell = 109$, let $M_0$ be the cubic subfield of $\Q(\mu_{\ell})$; then 
$\ell$ is inert in $K$ and  the $3$-class group of $L= KM_0$ is ${\sf [9,9,3]}$. 

\smallskip
If $\BJ_{L/K}$ is injective, $\Nu_{L/K}(\CH_L) = \BJ_{L/K}(\CH_K) \simeq \CH_K$ 
and $\CH_L^G \simeq \CH_K$ (same order); 
the formula (e.g., \cite[\S\,2, (1), (2)]{Gra2022} or \cite{Gra2017})
of the orders of the elements of the filtration of $\CH_L$ gives
$\order (\CH_L/\CH_L^G )^G = \ds \frac{\order \CH_K}{\order \Norm_{L/K}(\CH_L^G)} 
= 3^3$, since $r=1$ and $\Norm_{L/K}(\CH_L^G) =  \CH_K^3 = 1$, giving 
$\order \CH_L \geq 3^6$ (absurd). 

\smallskip
So, $\BJ_{L/K}$ is not injective.
This gives an example where:
\begin{equation*}
\left \{\begin{aligned}
\order \CH^\alg_L & := \order \{x \in \CH_L,\  \Nu_{L/K}(x) = 1\} \in \{3^3,3^4,3^5\} \\
\order \CH^\ar_L& := \order \{x \in \CH_L,\  \Norm_{L/K}(x) = 1\} = 3^2;
\end{aligned}\right.
\end{equation*}
so $\order \CH^\alg_L \! \times \order \CH^\alg_K \in \{3^6, 3^7,3^8\}$, 
while $\order \CH^\ar_L\! \times \order \CH^\ar_K = 3^2\! \times 3^3 = \order \CH_L$
giving illustration of Corollary \ref{relation}.

\smallskip
For $k = \Q(\sqrt{1129})$, whose class group is ${\sf [9]}$ and for the cubic field $K_0$ 
of conductor $7$, the $3$-class group of $K$ is ${\sf [27]}$. For  $\ell = 19$, let $M_0$ 
be the subfield of degree $9$ of $\Q(\mu_{\ell})$ and $L=KM_0$; then the $3$-class group 
of the subfield $K_1 \subset L$, of degree $3$ over $K$, is ${\sf [27]}$ (stability from $K$). 
From  \cite[Theorem 3.5]{Gra2022} this implies $\order \CH_L = 27$, $\CH_L^G 
= \CH_L$, giving a partial capitulation since $\Nu_{L/K}(\CH_L) = \CH_L^9$ of order $3$. 
Then $\order \CH^\alg_L  = 3^2$ and $\order \CH^\ar_L = 1$.

\smallskip
Recall that a {\it sufficient} condition of capitulation of $\CH_K$ in $L$ is the stability 
of the $p$-class groups from some layer in the $p$-tower $L = \bigcup_{0 \leq i \leq N} 
K_i$, where $[K_i : K] = p^i$, $[L : K] = p^N$ much bigger than the exponent of 
$\CH_K$ \cite[\S\,3.2]{Gra2022}. 

\smallskip
A more general study in \cite[Theorems 1.1, 1.2]{Gra2023} shows that capitulation is, 
surprisingly, inversely proportional to the ``complexity'' of $\CH_L$ in a precise 
meaning. The fact that, in huge $p$-towers $L/K$, stability may occur at some layer, is a 
reasonable conjecture as a kind of finite Iwasawa theory for which the first case of 
Greenberg's conjecture \cite[Theorem 1, \S\,4]{Gree1976} would hold.

\section{Checking of equalities
\texorpdfstring{$\order \CH^\ar_{K,\varphi} \!=
\order(\CE_K / \CE^\circ_K\, \CF_K)_{\varphi_0}$}{Lg} in the non semi-simple 
case \texorpdfstring{$p=3$, $[K:\Q]=6$}{Lg}}\label{appendix}

\subsection{Introduction}
To illustrate the RAMC, the minimal non trivial context may be for cyclic cubic 
fields $k$ and degree $7$ cyclic extensions $K_0$ with $K=k K_0$. 
Then, the checking of RAMC would be to find $\ell \equiv 1 \pmod {7^N}$ 
then $M_0 \subset \Q(\mu_\ell^{})$, of degree $7^n$, $n \leq N$, and to work 
in the compositum $L = M_0K$ of absolute degree $21\times 7^n$; this is 
oversized for PARI computations; but, as shown by examples of \S\,\ref{capitulations}, 
capitulations may occur easily.

\smallskip
Nevertheless, since the existence of (infinitely many) extensions 
$L/K$ in which $\CH_K$ capitulates, is a conjecture largely tested 
in \cite{Gra2022,Gra2023},
we limit ourselves to an easier non semi-simple case, that is to say, 
$p=3$, $k$ real quadratic and $K_0/\Q$ cubic cyclic; then we only 
check the relation of $\order \CH^\ar_{K,\varphi}$ with
the index $(\CE_K : \CE^\circ_K\, \CF_K)_\varphi$.

Many cases occur for the checking of $(\CE_K : \CE^\circ_K \CF_{\!K})_{\varphi_0}
= \order \CH^\ar_{K,\varphi}$ (Theorem \ref{chiformulaH}) when the structure 
of the class groups varies as well as the relations $\eta_i = \prod_{j=1}^3 
\varepsilon_j^{a_{i,j}}$, $1 \leq i \leq 3$, describing the relative
cyclotomic units $\eta_i$ from the relative units $\varepsilon_j$
(relative units meaning of norm $1$ over $k$ and $K_0$). 
Once again, we will see that the two $\Z_p[\CG_K]$-modules, of same order, 
$(\CE_K / \CE^\circ_K \CF_{\!K})_{\varphi_0}$ and $\CH^\ar_{K,\varphi}$, are not 
isomorphic in general. The program computes $\BH_k$ and $\BH_K$; 
so $\order \CH^\ar_{K,\varphi} = \order \CH_K \times  \order \CH_k^{-1}$.

\subsection{PARI program -- Examples}
To simplify (especially, for the computation of the cyclotomic units), we take
$k = \Q(\sqrt f)$, where $f$ is a prime number congruent to $1$ modulo $4$
and various cubic fields $K_0$ of prime conductors $q \equiv 1 \pmod 3$. 

\smallskip
In the program, the main parameters are the following:

\smallskip
${\sf Pk, PK}$: defining polynomials of ${\sf k, K}$;

\smallskip
${\sf Ck, CK}$: class groups of ${\sf k, K}$;
${\sf Ek, EK}$: unit groups of ${\sf k, K}$;

\smallskip
${\sf LEtaK}$: list of the $3$ independent relative cyclotomic units of ${\sf K}$.

\smallskip
The conductor ${\sf f}$ varies in ${\sf [bf, Bf]}$; the prime ${\sf q}$ varies in 
${\sf [bq, Bq]}$; then, ${\sf kronecker(f,q)}$ gives the splitting of $q$ in $k/\Q$,
which may give larger $3$-class group for $K$ if ${\sf kronecker(f,q)=1}$.

\smallskip
Then the computations of some indices of groups of units are done taking
the logarithms of the units, to get easier linear relations. We do not write
the units and their logarithms since they
are oversized; but the running of the program gives rapidly a complete data
(use large precision and memory); for instance, with ${\sf f=257}$ and
${\sf q=4597}$ a precision ${\sf \backslash p \ 300}$ is necessary for the 
computation of some logarithms.

\smallskip
For some rare cases cases, the program gives $5$ fundamental relative 
units $\varepsilon_i$ (instead of $3$), but with two relations of dependence;
this is due to PARI computation of the fundamental units in ${\sf K.fu}$
where the units of the cubic field $K_0$ do not appear as a direct sum; a same remark
may occur for the fundamental unit $\varepsilon_0$ of the quadratic field $k$ (in general, 
$\varepsilon_0 = \varepsilon_1$, but for $k=\Q(\sqrt {1129})$ and $q=7$, $\varepsilon_0 = 
\varepsilon_1\varepsilon_2^{-1}\varepsilon_3^{-1}$. The final matrix $\big( a_{i,j}\big)_{i,j}$ 
is such that $\eta_i = \prod_{j=1}^3 \varepsilon_j^{a_{i,j}}$. 
The PARI program uses the package \cite{Pari2016}:

\medskip
\ft\begin{verbatim}
\p 150
{p=3;BHK=3;bf=229;Bf=10^3;bq=7;Bq=10^4;
forprime(f=bf,Bf,if(Mod(f,4)!=1,next);Pk=x^2-f;k=bnfinit(Pk,1);
Hk=k.no;if(valuation(Hk,p)==0,next);Ck=k.clgp;
print();print("conductor k=",f," Pk=",Pk," classgroup(k)=",Ck[2]);
Ek=k.fu;Ak=lift(Ek[1]);Bk=polcoeff(Ak,0)+polcoeff(Ak,1)*sqrt(f);
LEke=2*log(abs(Bk));print("log of epsilon0=",LEke);
forprime(q=bq,Bq,if(Mod(q,3)!=1,next);Q=polsubcyclo(q,p);
PK=polcompositum(Pk,Q)[1];F=q*f;xK=real(polroots(PK)[2]);
K=bnfinit(PK,1);HK=K.no;if(valuation(HK,3)<BHK,next);print();
print("nombre premier q=",q," PK=",PK," classgroup(K)=",K.clgp[2]);
print("kronecker(f,q)=",kronecker(f,q));
\\computation of the generator of order 2 of Gal(K/Q):
GK=nfgaloisconj(K);Id=x;for(k=1,6,Z=GK[k];ks=1;while(Z!=Id,
Z=nfgaloisapply(K,GK[k],Z);ks=ks+1);if(ks==2,S2=GK[k];break));
\\computation of the relative units (of norm 1 over K0)
EKe=List;EK=K.fu;for(n=1,5,e=EK[n];
ee=nfgaloisapply(K,S2,e);e=e*ee^-1;listput(EKe,e));
LEKe=List;for(n=1,5,AK=lift(EKe[n]);BK=0;cK=0;
for(m=0,5,c=polcoeff(AK,m);BK=BK+xK^m*c);
listput(LEKe,log(abs(BK))));print();Ce=0;LEK=List;for(n=1,5,
e=LEKe[n];if(abs(abs(e))>10^-10,Ce=Ce+1;listput(LEK,e)));
print("logarithms of the relative units of K:",LEK);
g=lift(znprimroot(f));G=lift(znprimroot(q));
u=lift(Mod((1-g)/f,q));v=lift(Mod((1-G)/q,f));
g=g+u*f;G=G+v*q;g2=g^2;G3=G^3;d2=(f-1)/2;d3=(q-1)/3;
\\Computation of the Artin group A=Gal(Q(exp(I*Pi/F))/K):
z=exp(I*Pi/F);A=List;for(i=1,d2,for(j=1,d3/2,
a=Mod(g2,F)^i*Mod(G3,F)^j;a=lift(a);listput(A,a)));dA=d2*d3/2;
\\Computation of the relative cyclotomic units:
LEtaK=List;C=List;for(i=1,3,for(j=1,2,c=1;for(t=1,dA,a=A[t];
s=lift(Mod(a*G^i*g^j,F));c=c*(z^s-z^-s));c=real(c);listput(C,c)));
listput(LEtaK,log(abs(C[1]*C[2]^-1)));
listput(LEtaK,log(abs(C[3]*C[4]^-1)));
listput(LEtaK,log(abs(C[5]*C[6]^-1)));
print("Cyclotomic units:",LEtaK);if(Ce==5,print("Ce=5");next);
\\Matrix of relations:
BX=12;for(j=1,3,for(a=-BX,BX,for(b=-BX,BX,for(c=-BX,BX,
X=a*LEK[1]+b*LEK[2]+c*LEK[3];
if(abs(X-LEtaK[j])<10^-6,print(a," ",b," ",c))))))))}
\end{verbatim}\ns

\subsubsection{Quadratic field \texorpdfstring{$k$}{Lg} of conductor $229$}
${}$
\ft\begin{verbatim}
Pk=x^2-229   classgroup(k)=[3]
log(epsilon0)=-5.424930610368687949361759021212260139798716...
\end{verbatim}\ns
\ft\begin{verbatim}
(i) nombre premier q=37
PK=x^6+2*x^5-710*x^4-918*x^3+157031*x^2+125228*x-10725781
classgroup(K)=[3,3,3]  kronecker(f,q)=1  epsilon0=epsilon1
0  3  0
0  0 -3
0 -3  3
\end{verbatim}\ns

\smallskip\noindent
In the following case, the program gives $5$ relative units $\varepsilon'_i$
with the relations $\varepsilon'_3=\varepsilon'_2$ and $\varepsilon'_5=
\varepsilon'_2\cdot \varepsilon'^{-1}_4$, giving
the basis $\varepsilon_1= \varepsilon'_1$, $\varepsilon_2= \varepsilon'_2$,
$\varepsilon_3= \varepsilon'_4$:

\ft\begin{verbatim}
(ii) prime number q=1723 
PK=x^6+2*x^5-1834*x^4-7552*x^3+480853*x^2-252370*x-20931500
classgroup(K)=[9,3]  kronecker(f,q)=-1  epsilon0=epsilon1
1  0 -3
1  3  0
1 -3  3
\end{verbatim}\ns
\ft\begin{verbatim}
(iii) prime number q=5743
PK=x^6+2*x^5-4514*x^4-54942*x^3+3770063*x^2+62574440*x-31665625
classgroup(K)=[9,9]  kronecker(f,q)=-1  epsilon0=epsilon1
-2 -6 -3
 1  3 -3
 4  3  6
\end{verbatim}\ns
\ft\begin{verbatim}
(iv) prime number q=6379 
PK=x^6+2*x^5-4938*x^4+18458*x^3+4700367*x^2-32919224*x-679071697
classgroup(K)=[9,3,3]  kronecker(f,q)=1  epsilon0=epsilon1
-3  6  3
 3 -3 -6
 0 -3  3
\end{verbatim}\ns

\subsubsection{Quadratic field \texorpdfstring{$k$}{Lg} of conductor 
\texorpdfstring{$1129$}{Lg}}
${}$
\ft\begin{verbatim}
Pk=x^2-1129   class group(k)=[9]
log(epsilon0)=11.634240035110537169057276239224439429816675...
\end{verbatim}\ns

\noindent
In the next example, the fundamental unit $\varepsilon_0$ of $k$ is not $\varepsilon_1$
given by PARI; one gets $\varepsilon_0 = \varepsilon_1 \varepsilon_2^{-1}\varepsilon_3^{-1}$:

\ft\begin{verbatim}
(v) prime number q=7
PK=x^6+2*x^5-3390*x^4-4522*x^3+3821667*x^2+2547028*x-1432703257
classgroup(K)=[27]  kronecker(f,q)=1  epsilon0=epsilon1/epsilon2epsilon3
-1 -1  0
 0  1 -1
 1  0  1
\end{verbatim}\ns
\ft\begin{verbatim}
(vi) prime number q=73 
PK=x^6+2*x^5-3434*x^4-4618*x^3+3822187*x^2+2421872*x-1377322821
classgroup(K)=[189,3]  kronecker(f,q)=-1  epsilon0=epsilon1
 8 -6  9
 2 -3 -6
-1  9 -3
\end{verbatim}\ns
\ft\begin{verbatim}
(vii) nombre premier q=19867 
PK=x^6+2*x^5-16630*x^4-200242*x^3+47490067*x^2+607831028*x-25945199777 
class group K=[18,18,9]  kronecker(f,q)=-1  epsilon0=epsilon1
0  0 -9
9  9  9
0 -9  0
\end{verbatim}\ns

\noindent
In the following case, the program gives $5$ relative units $\varepsilon'_i$
with the relations $\varepsilon'_3=\varepsilon'_2$ and $\varepsilon'_5=
\varepsilon'_4\cdot \varepsilon'^{-1}_1 \cdot \varepsilon'_2$, giving
the basis $\varepsilon_1= \varepsilon'_1$, $\varepsilon_2= \varepsilon'_2$,
$\varepsilon_3= \varepsilon'_4$.

\ft\begin{verbatim}
(viii) prime number q=41077 
PK=x^6+2*x^5-30770*x^4+156750*x^3+191481179*x^2-1910572432*x-169077439885 
class group K=[27,9,3]  kronecker(f,q)=-1  epsilon0=epsilon1
-3 -9  9
 6  9  0
 6  0 -9
\end{verbatim}\ns

\subsection{Verification of the relation \texorpdfstring{$\order \CH^\ar_{K,\varphi} = 
(\CE_K : \CE^\circ_K \CF_{\!K})$}{Lg}.}

${}$

\smallskip\noindent
{\bf (a) Quadratic field $k$ of conductor $229$ ($\CH_k \simeq \Z/3\Z$):}

\medskip\noindent
{\bf Case (i):} ${\sf CH_K = [3,3,3]}$ and ${\sf CH_{K,\varphi} = [3,3]}$.
Since $\varepsilon_1 = \varepsilon_0$, the unit index is
\ft $\Big(\langle \varepsilon_1,\varepsilon_2,\varepsilon_3 \rangle :
\langle \varepsilon_1,\varepsilon_2^3,\varepsilon_3^{-3},
\varepsilon_2^{-3}\varepsilon_3^3 \rangle \Big)
= \Big(\langle \varepsilon_1, \varepsilon_2,\varepsilon_3 \rangle :
\langle \varepsilon_1,\varepsilon_2^3,\varepsilon_3^3 \rangle \Big) = 9.$\ns
 
The quotient $\CE_{K,\varphi} / \CE^\circ_{K,\varphi} \CF_{\!K,\varphi}$ being monogenic,
is isomorphic to $\Z/9 \Z$, wile $\CH_{K,\varphi} \simeq \Z/3\Z \times \Z/3\Z$.

\smallskip\noindent
{\bf Case (ii):} ${\sf CH_K = [9,3]}$ and 
${\sf CH_{K,\varphi} = [3,3]}$ or ${\sf [9]}$. Once we are using
the basis $\varepsilon_1= \varepsilon'_1$, $\varepsilon_2= \varepsilon'_2$,
$\varepsilon_3= \varepsilon'_4$, the matrix is analogous to the previous one, 
whence the result with the index $9$.

\smallskip\noindent
{\bf Case (iii):} ${\sf CH_K = [9,9]}$ and ${\sf CH_{K,\varphi} = [9,3]}$.
The unit index is:
\ft\begin{equation*}
\begin{aligned}
& \Big(\langle \varepsilon_1,\varepsilon_2,\varepsilon_3 \rangle : 
\langle \varepsilon_1, \varepsilon_1^{-2} \varepsilon_2^{-6}\varepsilon_3^{-3},
\varepsilon_1 \varepsilon_2^3 \varepsilon_3^{-3},
\varepsilon_1^4\varepsilon_2^{3}\varepsilon_3^6 \rangle \Big)= \\
& \Big(\langle \varepsilon_1, \varepsilon_2, \varepsilon_3 \rangle :
\langle \varepsilon_1,\varepsilon_2^3 \varepsilon_3^{-3}, 
\varepsilon_2^9, \varepsilon_3^9 \rangle \Big) 
=  \Big(\langle \varepsilon_1, \varepsilon_2, \varepsilon_3 \rangle :
\langle \varepsilon_1,\varepsilon_2^3 \varepsilon_3^{-3}, 
\varepsilon_2^9 \rangle \Big)= 27.
\end{aligned}
\end{equation*}\ns

\smallskip\noindent
{\bf Case (iv):} ${\sf CH_K = [9,3,3]}$ and ${\sf CH_{K,\varphi} = [9,3]}$ or ${\sf [3,3,3]}$:
\ft\begin{equation*}
\begin{aligned}
& \Big(\langle \varepsilon_1,\varepsilon_2,\varepsilon_3 \rangle : 
\langle \varepsilon_1, \varepsilon_1^{-3} \varepsilon_2^{6}\varepsilon_3^{3},
\varepsilon_1^{3} \varepsilon_2^{-3} \varepsilon_3^{-6},
\varepsilon_2^{-3}\varepsilon_3^3 \rangle \Big)= \\
 &  \Big(\langle \varepsilon_1, \varepsilon_2, \varepsilon_3 \rangle :
\langle \varepsilon_1,\varepsilon_2^{-3} \varepsilon_3^{3}, 
\varepsilon_2^9, \varepsilon_3^9 \rangle \Big) 
=  \Big(\langle \varepsilon_1, \varepsilon_2, \varepsilon_3 \rangle :
\langle \varepsilon_1,\varepsilon_2^{-3} \varepsilon_3^{3}, 
\varepsilon_2^9 \rangle \Big)= 27.
\end{aligned}
\end{equation*}\ns

\medskip\noindent
{\bf (b) Quadratic field $k$ of conductor $1129$ ($\CH_k \simeq \Z/9\Z$):}

\medskip\noindent
{\bf Case (v):} ${\sf CH_K = [27]}$ and ${\sf CH_{K,\varphi} = [3]}$.
We have  $\varepsilon_0 = \varepsilon_1 \varepsilon_2^{-1} \varepsilon_3^{-1}$:
\ft\begin{equation*}
\begin{aligned}
&\Big(\langle 
\varepsilon_1 \varepsilon_2^{-1} \varepsilon_3^{-1},\varepsilon_2,\varepsilon_3 \rangle : 
\langle \varepsilon_1 \varepsilon_2^{-1} \varepsilon_3^{-1},
\varepsilon_1^{-1} \varepsilon_2^{-1},
\varepsilon_2 \varepsilon_3^{-1},
\varepsilon_1 \varepsilon_3 \rangle \Big)= \\
& \Big(\langle \varepsilon_1 \varepsilon_2^{-1} \varepsilon_3^{-1},\varepsilon_2,\varepsilon_3 \rangle :
\langle \varepsilon_1 \varepsilon_2^{-1} \varepsilon_3^{-1},
\varepsilon_2 \varepsilon_3^{-1}, 
\varepsilon_1^{-1} \varepsilon_2^{-1}, \varepsilon_1 \varepsilon_3\rangle \Big) 
=  \Big(\langle \varepsilon_1 \varepsilon_2^{-1} \varepsilon_3^{-1},
\varepsilon_1 \varepsilon_2,\varepsilon_1 \varepsilon_3 \rangle \Big)= 3.
\end{aligned}
\end{equation*}\ns

\smallskip\noindent
{\bf Case (vi):} ${\sf CH_K = [27,3]}$ and ${\sf CH_{K,\varphi} = [3,3]}$:
\ft\begin{equation*}
\begin{aligned}
& \Big(\langle \varepsilon_1,\varepsilon_2,\varepsilon_3 \rangle : 
\langle \varepsilon_1, \varepsilon_1^{8} \varepsilon_2^{-6}\varepsilon_3^{9},
\varepsilon_1^{2} \varepsilon_2^{-3} \varepsilon_3^{-6},
\varepsilon_1^{-1}\varepsilon_2^{9}\varepsilon_3^{-3} \rangle \Big)= \\
& \Big(\langle \varepsilon_1, \varepsilon_2, \varepsilon_3 \rangle :
\langle \varepsilon_1, \varepsilon_2^{-6} \varepsilon_3^{9}, 
\varepsilon_2^{-3} \varepsilon_3^{-6},\varepsilon_2^9 \varepsilon_3^{-3} \rangle \Big) 
=  \Big(\langle \varepsilon_1, \varepsilon_2, \varepsilon_3 \rangle :
\langle \varepsilon_1, \varepsilon_2^3 \varepsilon_3^6,
\varepsilon_2^9 \varepsilon_3^{-3} \rangle \Big) = 9.
\end{aligned}
\end{equation*}\ns

\smallskip\noindent
{\bf Case (vii):} ${\sf CH_K = [9,9,9]}$ and ${\sf CH_{K,\varphi} = [9,9]}$:
\ft\begin{equation*}
 \Big(\langle \varepsilon_1,\varepsilon_2,\varepsilon_3 \rangle : 
\langle \varepsilon_1, \varepsilon_3^{-9},
\varepsilon_1^{9} \varepsilon_2^{9}\varepsilon_3^{9},
\varepsilon_2^{-9} \rangle \Big)= 
 \Big(\langle \varepsilon_1, \varepsilon_2, \varepsilon_3 \rangle :
\langle \varepsilon_1, \varepsilon_2^{9}, \varepsilon_3^{9} \rangle \Big) = 81.
\end{equation*}\ns

\smallskip\noindent
{\bf Case (viii):} ${\sf CH_K = [27,9,3]}$ and ${\sf CH_{K,\varphi} = [27,3]}$ or ${\sf [3,9,3]}$:
\ft\begin{equation*}
\begin{aligned}
& \Big(\langle \varepsilon_1,\varepsilon_2,\varepsilon_3 \rangle : 
\langle \varepsilon_1, \varepsilon_1^{-3} \varepsilon_2^{-9}\varepsilon_3^{9},
\varepsilon_1^{6} \varepsilon_2^{9},
\varepsilon_1^{6} \varepsilon_3^{-9} \rangle \Big)= \\
& \Big(\langle \varepsilon_1, \varepsilon_2, \varepsilon_3 \rangle :
\langle \varepsilon_1, \varepsilon_2^{-9} \varepsilon_3^{9}, 
\varepsilon_2^{9}, \varepsilon_3^{9} \rangle \Big) 
=  \Big(\langle \varepsilon_1, \varepsilon_2, \varepsilon_3 \rangle :
\langle \varepsilon_1, \varepsilon_2^9, \varepsilon_3^9 \rangle \Big) = 81.
\end{aligned}
\end{equation*}\ns

\section{Conclusion}
It is not possible, with PARI, to check the RAMC with non-trivial examples such as $k$ 
cubic, $[K_0 : \Q] = 7$, with $7$-extensions $L/K$ of degree $7^N$ in which 
$\CH_K$ capitulates; nevertheless, we have given examples of non-trivial 
capitulations for $k$ quadratic, $[K_0 : \Q] = 3$ and cubic extensions $L/K$, which
may be sufficient to suggest that the method applies in any circumstance. 
Indeed, we think that the existence of capitulation fields $L \subset K(\mu_\ell)$, with
$\ell \equiv 1 \pmod {2p^N}$ inert in $K$, is a governing conjecture for  many 
arithmetical properties; even if its proof seems out of reach, it has been checked 
on several other occasions, and may be considered as a kind of ``basic conjecture'' for 
class field theory, that is to say a deep diophantine property, predominant 
in a logical point of view.

\end{document}